\documentclass[runningheads,english,12pt]{article}
\usepackage{amsmath,amssymb,amsbsy,amsfonts,amsthm,latexsym,amsopn,amstext,
                                               amsxtra,euscript,amscd}

\newfont{\teneufm}{eufm10}
\newfont{\seveneufm}{eufm7}
\newfont{\fiveeufm}{eufm5}
\newfam\eufmfam
                \textfont\eufmfam=\teneufm \scriptfont\eufmfam=\seveneufm
                \scriptscriptfont\eufmfam=\fiveeufm
\def\bbbc{{\mathchoice {\setbox0=\hbox{$\displaystyle\rm C$}\hbox{\hbox
to0pt{\kern0.4\wd0\vrule height0.9\ht0\hss}\box0}}
{\setbox0=\hbox{$\textstyle\rm C$}\hbox{\hbox
to0pt{\kern0.4\wd0\vrule height0.9\ht0\hss}\box0}}
{\setbox0=\hbox{$\scriptstyle\rm C$}\hbox{\hbox
to0pt{\kern0.4\wd0\vrule height0.9\ht0\hss}\box0}}
{\setbox0=\hbox{$\scriptscriptstyle\rm C$}\hbox{\hbox
to0pt{\kern0.4\wd0\vrule height0.9\ht0\hss}\box0}}}}
\def\bbbq{{\mathchoice {\setbox0=\hbox{$\displaystyle\rm
Q$}\hbox{\raise 0.15\ht0\hbox to0pt{\kern0.4\wd0\vrule
height0.8\ht0\hss}\box0}} {\setbox0=\hbox{$\textstyle\rm
Q$}\hbox{\raise 0.15\ht0\hbox to0pt{\kern0.4\wd0\vrule
height0.8\ht0\hss}\box0}} {\setbox0=\hbox{$\scriptstyle\rm
Q$}\hbox{\raise 0.15\ht0\hbox to0pt{\kern0.4\wd0\vrule
height0.7\ht0\hss}\box0}} {\setbox0=\hbox{$\scriptscriptstyle\rm
Q$}\hbox{\raise 0.15\ht0\hbox to0pt{\kern0.4\wd0\vrule
height0.7\ht0\hss}\box0}}}}
\def\bbbt{{\mathchoice {\setbox0=\hbox{$\displaystyle\rm
T$}\hbox{\hbox to0pt{\kern0.3\wd0\vrule height0.9\ht0\hss}\box0}}
{\setbox0=\hbox{$\textstyle\rm T$}\hbox{\hbox
to0pt{\kern0.3\wd0\vrule height0.9\ht0\hss}\box0}}
{\setbox0=\hbox{$\scriptstyle\rm T$}\hbox{\hbox
to0pt{\kern0.3\wd0\vrule height0.9\ht0\hss}\box0}}
{\setbox0=\hbox{$\scriptscriptstyle\rm T$}\hbox{\hbox
to0pt{\kern0.3\wd0\vrule height0.9\ht0\hss}\box0}}}}
\def\bbbs{{\mathchoice
{\setbox0=\hbox{$\displaystyle     \rm S$}\hbox{\raise0.5\ht0\hbox
to0pt{\kern0.35\wd0\vrule height0.45\ht0\hss}\hbox
to0pt{\kern0.55\wd0\vrule height0.5\ht0\hss}\box0}}
{\setbox0=\hbox{$\textstyle        \rm S$}\hbox{\raise0.5\ht0\hbox
to0pt{\kern0.35\wd0\vrule height0.45\ht0\hss}\hbox
to0pt{\kern0.55\wd0\vrule height0.5\ht0\hss}\box0}}
{\setbox0=\hbox{$\scriptstyle      \rm S$}\hbox{\raise0.5\ht0\hbox
to0pt{\kern0.35\wd0\vrule height0.45\ht0\hss}\raise0.05\ht0\hbox
to0pt{\kern0.5\wd0\vrule height0.45\ht0\hss}\box0}}
{\setbox0=\hbox{$\scriptscriptstyle\rm S$}\hbox{\raise0.5\ht0\hbox
to0pt{\kern0.4\wd0\vrule height0.45\ht0\hss}\raise0.05\ht0\hbox
to0pt{\kern0.55\wd0\vrule height0.45\ht0\hss}\box0}}}}
\def\bbbz{{\mathchoice {\hbox{$\sf\textstyle Z\kern-0.4em Z$}}
{\hbox{$\sf\textstyle Z\kern-0.4em Z$}} {\hbox{$\sf\scriptstyle
Z\kern-0.3em Z$}} {\hbox{$\sf\scriptscriptstyle Z\kern-0.2em
Z$}}}}

\newtheorem{theorem}{Theorem}
\newtheorem{lemma}[theorem]{Lemma}

\def\squareforqed{\hbox{\rlap{$\sqcap$}$\sqcup$}}
\def\qed{\ifmmode\squareforqed\else{\unskip\nobreak\hfil
\penalty50\hskip1em\null\nobreak\hfil\squareforqed
\parfillskip=0pt\finalhyphendemerits=0\endgraf}\fi}

\def\cE{{\mathcal E}}
\def\cF{{\mathcal F}}

\def\cU{{\mathcal U}}

\def\cW{{\mathcal W}}

\newcommand{\ignore}[1]{}

\def\vec#1{\mathbf{#1}}

\def\lcm{{\mathrm{lcm}\,}}

\def \F{\mathbb{F}}

\def \Q{\mathbb{Q}}

\def\\{\cr}
\def\({\left(}
\def\){\right)}

\begin{document}

\title{Degree Growth,  Linear Independence and Periods 
of a Class of Rational Dynamical Systems}

\author{{\sc Alina Ostafe} \\
{Department of Computing, Macquarie University} \\
{Sydney, NSW 2109, Australia} \\
{\tt alina.ostafe@mq.edu.au}\\
\and
{\sc Igor Shparlinski} \\
{Department of Computing, Macquarie University} \\
{Sydney, NSW 2109, Australia} \\
{\tt igor.shparlinski@mq.edu.au}  
}

\maketitle
\begin{abstract} We introduce and study algebraic 
dynamical systems generated by triangular systems 
of rational functions. We obtain several  results 
about the degree growth and linear independence 
of iterates as well as about  possible lengths
of trajectories generated by such dynamical systems over 
finite fields.  Some of these results are generalisations
of those known in the polynomial case, some are new even
in this case.
\end{abstract}

\paragraph{\bf MSC(2010):} 
Primary 37P05; Secondary 11T06, 37P25, 65C10 
\section{Introduction}

Let $\F$ be an arbitrary field $\F$ and let $F_1,\ldots,F_m \in \F(X_1, \ldots, X_m)$
be  $m$ rational functions in $m$ variables over $\F$. 
For each $i=1, \ldots ,m$ we define the $k$-th iteration of the rational function $F_i$ by the recurrence relation
\begin{equation}
\label{eq:PolyIter}
F_i^{(0)}=X_i, \quad F_i^{(k)}= F_i\(F_1^{(k-1)}, \ldots ,F_m^{(k-1)}\), \quad 
k=1,2,\ldots\,  .
\end{equation}
In this paper we consider dynamical systems generated by multivariate rational functions, we refer 
to~\cite{AnKhr,Schm,Silv1} for a background on algebraic 
dynamical systems.

More precisely, we define the vectors
$\vec{u}_n =(u_{n,1} ,\ldots,u_{n,m})\in \F^m$
by the recurrence relation  
\begin{equation}
\label{eq:Gen}
u_{n+1,i}= F_i(u_{n,1},\ldots,u_{n,m}), \qquad n =0,1,\ldots, \quad i=1,\ldots,m,
\end{equation}
with some {\em initial vector} $\vec{u}_0 = (u_{0,1},\ldots,u_{0,m})\in \F^{m}$. 

As we work with rational functions, we make the standard 
convention (see~\cite{FN,NiSh2,NiSh3}) that 
\begin{equation}
\label{eq:Zero}
0^{-1}=0.
\end{equation}

Using the following vector notation
$$
\vec{F}=(F_1(X_1,\ldots,X_m),\ldots,F_m(X_1,\ldots,X_m)),
$$
we have the recurrence relation
\begin{equation}
\label{eq:Vec}
\vec{u}_{n+1}=\vec{F}(\vec{u}_{n}),
\quad n=0,1,\ldots.
\end{equation}
In particular, for any $n\ge 0$ and $i=1,\ldots,m$ we have
$$
u_{n,i} = F_i^{(n)}(\vec{u}_{0}) = F_i^{(n)}(u_{0,1},\ldots,u_{0,m})
$$
or $$\vec{u}_{n}=\vec{F}^{(n)}(\vec{u}_0),$$
provided that $\vec{u}_{n}$ has been  generated by~\eqref{eq:Vec} without 
using the convention~\eqref{eq:Zero}
(that is, no poles have been encountered).

Clearly, if we work over a finite field of $q$ 
elements, the above sequence~\eqref{eq:Vec} of vectors 
$\{\vec{u}_{n}\}$ is eventually periodic
with some period $\tau \le q^{m}$.

One of the important characteristics of the dynamical system 
generated by $F_1,\ldots,F_m \in \F(X_1, \ldots, X_m)$
is the degree growth of the functions~\eqref{eq:PolyIter}.
It  is of 
great interest for the theory of dynamical systems and has been studied
in a number of works, see, for example,~\cite{BeTr,Via} and references therein.
It is also important 
for applications to pseudorandom number generators~\cite{TopWin}. 

More precisely, 
although for a ``typical'' system an exponential degree growth is expected, 
there are several examples of systems where the degree 
grows much slower (which is highly beneficial for their applications), 
and such systems are of special interest.

For example, in~\cite{Ost2,OstShp1}  
several types of multivariate polynomial systems $\cF = \{F_1, \ldots ,F_{m}\}$ of $m$ polynomials 
in $m$ variables over a
finite field $\F_q$  have been constructed and studied,  having the ``triangular" form
\begin{equation}
\label{eq:syst}
\begin{split}
&F_1(X_1, \ldots ,X_m)=X_1G_1(X_2,\ldots,X_m)+H_1(X_2,\ldots,X_m),\\
  &\ldots  \\
&F_{m-1}(X_1, \ldots ,X_m)=X_{m-1}G_{m-1}(X_m)+H_{m-1}(X_m),\\
&F_m(X_1, \ldots ,X_m) = g_mX_m+h_m,
\end{split}
\end{equation}
with $G_i, H_i\in\F_q[X_{i+1},\ldots,X_m]$, $i=1,\ldots,m-1$, and $g_m, h_m\in\F_q$, $g_m\ne0$.
These systems have been further investigated in~\cite{Ost1,OstShp2,OstShpWin1,OstShpWin2}.

For the systems~\eqref{eq:syst}, 
in the case of constant polynomials $G_i  \in \F_q^*$ in~\cite{Ost2}
and polynomials $G_i$ with leading terms of special form in~\cite{OstShp1,OstShp2}, 
a series of results have been obtained about the distribution of the
corresponding sequences given by~\eqref{eq:Gen} 
that are much stronger than those known for 
generic systems. Moreover, for these classes of polynomials, it
has been shown in~\cite{OstShp1} that the degrees of the 
 iterations of the polynomials $F_i$, $i=1,\ldots,m$, grow 
significantly slower than the exponential growth  expected 
for the iterations of a ``generic'' system of  $m$ polynomials 
in $m$ variables. 
In turn, this leads (see~\cite{OstShp2}) to 
 much better estimates  
of exponential sums, and thus of discrepancy, for vectors 
generated by~\eqref{eq:syst} than for those 
originated from arbitrary polynomial systems (see~\cite{GNS,GG,OPS}).

We also note that the results obtained in~\cite{OstShp1,OstShp2} regarding the degree growth of the iterations of the polynomials in~\eqref{eq:syst} hold over any field $\F$.

In this paper we extend the class of rational dynamical systems with 
slow degree growth 
and present  an analogue of the  construction~\eqref{eq:syst}, 
but with rational functions defined by
\begin{equation}
\label{eq:ratsyst}
\begin{split}
&F_1(X_1, \ldots ,X_m)=X_1^{e_1}G_1(X_2,\ldots,X_m)+H_1(X_2,\ldots,X_m),\\
  &\ldots  \\
&F_{m-1}(X_1, \ldots ,X_m)=X_{m-1}^{e_{m-1}}G_{m-1}(X_m)+H_{m-1}(X_m),\\
&F_m(X_1, \ldots ,X_m) = g_mX_m^{e_m}+h_m,
\end{split}
\end{equation}
with $e_1,\ldots,e_m\in\{-1,1\}$, $G_i, H_i\in\F[X_{i+1},\ldots,X_m]$, $i=1,\ldots,m-1$, and $g_m, h_m\in\F$, $g_m\ne0$.

We note that for $m=1$ and $e=1$ we obtain the classical linear congruential generator which have been successully 
used for decades in the theory of Quasi Monte Carlo methods, 
see~\cite{Nied1,Nied2}, and for $m=1$ and $e=-1$, the classical inversive generator, see~\cite{NiSh2,NiSh3,NiSh4,NiWi}.

For the above class of multivariate rational functions, we study the 
degree growth under iterations and, using an 
approach similar to that  of~\cite[Lemma~1]{OstShp1},  we show in 
Section~\ref{sec:Deg}
that under certain additional conditions imposed on the 
systems of rational functions~\eqref{eq:ratsyst}, the degree grows polynomially.

Moreover, for applications to pseudorandom number generators, following the standard technique almost identical to that of~\cite{OstShp1},  one almost immediately obtains bounds on the exponential sums with elements of the sequence~\eqref{eq:Vec} generated by the 
system~\eqref{eq:ratsyst} (satisfying the conditions 
outlined in Section~\ref{sec:Deg}),  that in turn
leads to estimates on the uniformity of  
distribution of the vectors~\eqref{eq:Vec}.
However, one has also to prove that for any $k\ne l$ and
nonzero 
vector $\vec{a} = (a_1, \ldots, a_{m-1}) \in \F^{m-1}$, the linear combination
\begin{equation}
\label{eq:lincomb}
Q_{k,l,\vec{a}}=\sum_{i=1}^{m-1} a_i(F_i^{(k)}-F_i^{(l)})
\end{equation}
is a non-constant rational function.  We note that in the case of rational functions this does not follow directly from the degree argument as in the case of the polynomial systems~\eqref{eq:syst} in~\cite{OstShp1}, but we give such a result in 
Section~\ref{sec:LinIndep}.

 Since the derivation of such bounds of exponential sums
for our systems does not bring anything new to the area, we do not do this here 
but rather concentrate on the study of the degree and linear independence of 
iterates, which is also of interest for the general area of algebraic dynamics.

Furthermore, we   consider a related question about the length of trajectories generated 
by iterations~\eqref{eq:Vec} over a finite field $\F_q$. We remark that in this case a trajectory falls into a cycle if  $\vec{u}_{t} = \vec{u}_{s}$ for some integers $t > s \ge 0$. 
In particular, we   show that under some rather broad conditions 
for any fixed $\varepsilon >0$, for all but $o(q^{m})$ initial 
vectors $\vec{u}_0\in \F_q^m$, the trajectory length $t$ 
of the iterations~\eqref{eq:Vec} is at least $q^{1/3 -\varepsilon}$.

We note that Silverman~\cite{Silv2} has considered a question about periods of general polynomial systems but in somewhat dual situation when 
the initial value is fixed and the iterations are considered 
over a family of finite fields. The results of~\cite{Silv2}
apply to very general systems, however the estimates are 
only logarithmic rather than a power of the field size.

Moreover, we give necessary and sufficient conditions for the 
systems~\eqref{eq:ratsyst} to generate sequences of maximal period. We note that for the case $e_i=1$ for all $i=1,\ldots,m$, the maximal period length of the sequence generated by the system~\eqref{eq:syst} is achieved whenever the conditions of~\cite[Theorem~6]{Ost3} are satisfied.
Our result is a generalisation of that of~\cite{Ost3}. 

\section{Structure of the Iterations}

As in~\cite{OstShp1}, we can describe explicitly the iterations of the rational functions $F_i$ 
as follows.

Let us define the sets
\begin{equation}
\label{eq:I+-}
I_+=\{1\le i\le m~:~e_i=1\}\quad \text{and}
\quad I_-=\{1\le i\le m~:~e_i=-1\}.
\end{equation}

We also define 
\begin{eqnarray*}
G_i^{(\ell)} (X_{i+1},\ldots,X_m)&=&G_i\(F_{i+1}^{(\ell-1)},\ldots,F_m^{(\ell-1)}\),\\
H_i^{(\ell)} (X_{i+1},\ldots,X_m)&=&H_i\(F_{i+1}^{(\ell-1)},\ldots,F_m^{(\ell-1)}\).
\end{eqnarray*}

\begin{lemma}
\label{lem:LinTermSyst} 
Let $F_1, \ldots, F_m$ be rational functions defined by~\eqref{eq:ratsyst}. Then for  $i=1,\ldots,m-1$ 
and $k=0,1,\ldots$,
for the rational functions $F_i^{(k)}$ 
given by~\eqref{eq:PolyIter},

\begin{enumerate}
\item for every $i\in I_+$, $i < m$, we have
\begin{equation}
\label{eq:Fk1}
F_i^{(k)} = X_iG_{i,k} +H_{i,k},
\end{equation}
where $G_{i,k}, H_{i,k}\in\F(X_{i+1},\ldots,X_m)$ are defined by
\begin{equation}
\begin{split}
\label{eq:GH}
G_{i,k}&=G_iG_i^{(2)}\ldots G_i^{(k)},\\
H_{i,k} &=H_iG_i^{(2)}\ldots G_i^{(k)}+H_i^{(2)}G_i^{(3)}\ldots G_i^{(k)}+\ldots+H_i^{(k-1)}G_i^{(k)}+H_i^{(k)};
\end{split}
\end{equation} 
\item for every $i\in I_-$, $i < m$, we have:
\begin{equation}
\label{eq:Fk-1}
F_i^{(k)} = \frac{X_iR_{i,k} +S_{i,k}}{X_iR_{i,k-1} +S_{i,k-1}}, 
\end{equation}
where $R_{i,k},S_{i,k}$ are defined by the recurrence relations
\begin{equation}
\label{eq:RS}
\begin{split}
R_{i,k}&=G_i^{(k)}R_{i,k-2}+H_i^{(k)}R_{i,k-1}\\
S_{i,k}&=G_i^{(k)}S_{i,k-2}+H_i^{(k)}S_{i,k-1}
\end{split}
\end{equation} 
for $k\ge 1$, with the initial rational functions
 $$
 R_{i,0}=1,\quad S_{i,0}=0,\quad R_{i,1}=H_i,\quad S_{i,1}=G_i;
 $$

\item if $m \in I_+$, then 
$$
F_m^{(k)}=g_m^kX_m+(g_m^{k-1}+\ldots+g_m+1)h_m;
$$
\item if $m \in I_-$, then 
$$
F_m^{(k)}=\frac{(A^k)_{1,1}X_m+(A^k)_{1,2}}{(A^k)_{2,1}X_m+(A^k)_{2,2}},
$$
where 
$$
A^k=\(\begin{array}{cc}h_m & 
g_m\\
1&0\end{array}\)^k = \(\begin{array}{cc}(A^k)_{1,1} & 
(A^k)_{1,2}\\
(A^k)_{2,1}&(A^k){2,2}\end{array}\).
$$
\end{enumerate}
\end{lemma}

\begin{proof}
The case $e_i=1$, $i=1,\ldots,m$, is given by~\cite[Lemma 1]{OstShp2}. We consider now that $e_i=-1$ and prove the result by induction on the number of iterations $k$. For $k=1$ it is clear from the definition of the system, so we consider the statement true for the first $k-1$ iterations and we prove it for the $k$-th iteration. For $i=1,\ldots,m-1$, we have
\begin{equation*}
\begin{split}
F_i^{(k)}&=F_i(F_{i}^{(k-1)},F_{i+1}^{(k-1)},\ldots,F_m^{(k-1)})\\
&=\frac{F_{i}^{(k-1)}H_i^{(k)}+G_i^{(k)}}{F_{i}^{(k-1)}}=  \frac{\frac{X_iR_{i,k-1} +S_{i,k-1}}{X_iR_{i,k-2} +S_{i,k-2}}H_i^{(k)}+G_i^{(k)}}{\frac{X_iR_{i,k-1} +S_{i,k-1}}{X_iR_{i,k-2} +S_{i,k-2}}}\\
&= \frac{X_i(G_i^{(k)}R_{i,k-2}+H_i^{(k)}R_{i,k-1})+G_i^{(k)}S_{i,k-2}+H_i^{(k)}S_{i,k-1}}{X_iR_{i,k-1} +S_{i,k-1}}
\end{split}
\end{equation*} 
and thus we conclude this case. 
When $e_m=-1$, it is also clear as the $k$-th iteration of 
$$
F_m=\frac{h_mX_m+g_m}{X_m}
$$ is given by $A^k$ as simple  calculations show.
\end{proof}

We want to describe the degree growth of the iterations of the rational functions defined by~\eqref{eq:ratsyst}, and in particular to prove that we have the same effect of slow degree growth as for the polynomial systems~\eqref{eq:syst} described in~\cite[Lemma 1]{OstShp1}. 
To be able to give an explicit formula for the degree growth we need to impose some further conditions on the degrees of the polynomials $G_i$ and $H_i$, $i=1,\ldots,m-1$.

Let $F_1, \ldots, F_m$ be rational functions defined by~\eqref{eq:ratsyst}. From now on we  consider the system~\eqref{eq:ratsyst} satisfying the following conditions for $F_i$ for any $i=1,\ldots,m$: 
\begin{enumerate}
\item if $e_i=1$, as in~\cite{OstShp1,OstShp2}, we assume that the polynomial $G_i$ has a unique leading monomial $X_{i+1}^{s_{i,i+1}}\ldots X_m^{s_{i,m}}$, that is
$$G_i=g_iX_{i+1}^{s_{i,i+1}}\ldots X_m^{s_{i,m}}+ \widetilde{G}_i,$$
where  $g_i\in\F^*$ and $\widetilde{G}_i\in\F[X_{i+1},\ldots,X_m]$ with
\begin{equation}
\label{eq:deg1}
\deg_{X_j} \widetilde{G}_i<s_{i,j},\quad \deg_{X_j} H_i<s_{i,j},\quad j=i+1,\ldots,m;
\end{equation}
\item if $e_i=-1$, we assume that the polynomial $H_i$ has a unique leading monomial $X_{i+1}^{s_{i,i+1}}\ldots X_m^{s_{i,m}}$, that is
$$H_i=h_iX_{i+1}^{s_{i,i+1}}\ldots X_m^{s_{i,m}}+ \widetilde{H}_i,$$
where $h_i\in\F^*$ and $\widetilde{H}_i\in\F[X_{i+1},\ldots,X_m]$, and \begin{equation}
\label{eq:deg2}
\deg_{X_j} \widetilde{H}_i<s_{i,j},\quad \deg_{X_j} G_i<2s_{i,j},\quad j=i+1,\ldots,m.
\end{equation}
\end{enumerate}

We note that having these conditions   also allows us to consider the rational function system with constant multipliers $G_i$, 
$i=1,\ldots,m-1$. We remark that in~\cite{Ost2}, the case of constant polynomials $G_i$, $i=1,\ldots,m-1$, in the system~\eqref{eq:syst} was considered, but this case is different from the case of rational functions as the conditions on the degrees also differ, see~\eqref{eq:deg1} and~\eqref{eq:deg2}. Having this, we prove the following formula for the degree growth which coincides with~\cite[Lemma~1]{OstShp1}.

\section{Degree Growth}
\label{sec:Deg}

\begin{theorem}
\label{thm:deg}
Let $F_1, \ldots, F_m$ be rational functions defined by~\eqref{eq:ratsyst} satisfying the conditions~\eqref{eq:deg1} and~\eqref{eq:deg2} and
such that $s_{i,i+1}\ne 0$, $i=1,\ldots,m-1$. Then the degrees of the iterations of $F_1, \ldots ,F_m$ 
grow as follows
  \begin{eqnarray*} \deg F_i^{(k)}&=&\frac{1}{(m-i)!}k^{m-i}s_{i,i+1}\ldots s_{m-1,m}+\psi_i(k),\qquad  i=0,\ldots, m-1,\\
\deg F_m^{(k)}&=&1, 
\end{eqnarray*} 
where $\psi_i(T) \in \Q[T]$ is a polynomial of degree  $\deg \psi_i <m-i$.
\end{theorem}

\begin{proof} The proof is based on Lemma~\ref{lem:LinTermSyst}. 
The case $e_i=1$, using~\eqref{eq:Fk1} and~\eqref{eq:Fk-1}, follows exactly the same as in~\cite[Lemma~1]{OstShp1}.

We prove now the case when $e_i=-1$. Using the conditions~\eqref{eq:deg1} and~\eqref{eq:deg2} and the recurrence relation~\eqref{eq:RS}, it is easy to see that
\begin{equation}
\label{eq:deg3}
\begin{split}
\deg F_i^{(k)}&=\deg R_{i,k}+1=\deg H_i^{(k)}R_{i,k-1}+1\\
&=\deg H_iH_i^{(2)}\ldots H_i^{(k)}+1=\sum_{j=1}^k\deg H_i^{(j)}+1.
\end{split}
\end{equation}
As in~\cite[Lemma~1]{OstShp1}, we use induction on the number of variables $m$. For $m=2$ we easily see that $\deg H_1^{(j)}=\deg H_1=s_{1,2}$, and thus, by~\eqref{eq:deg3}, we have that $\deg F_1^{(k)}=k s_{1,2}+1$. 
We assume now that the theorem is true 
for $m-1$ variables and we prove it for $m$. For any $i=1,\ldots,m-1$, by the induction hypothesis, we have
\begin{equation*}
\begin{split}
\deg F_i^{(k)}&=\sum_{j=1}^k\deg H_i^{(j)}+1=\sum_{j=1}^k\deg H_i(F_{i+1}^{(j-1)},\ldots,F_m^{(j-1)})+1\\
&=\sum_{j=1}^k\deg \((F_{i+1}^{(j-1)})^{s_{i,i+1}}\ldots (F_m^{(j-1)})^{s_{i,m}}\)+1\\
&=\sum_{j=1}^k\left( \frac{1}{(m-i-1)!}(j-1)^{m-i-1}s_{i,i+1}s_{i+1,i+2}\ldots s_{m-1,m}+\right.\\ 
& \qquad\qquad\qquad\qquad\qquad\qquad\qquad \ldots\left.+(j-1)s_{i,m}s_{m-1,m}\right)+1.
\end{split}
\end{equation*}
As
$$
\sum_{j=1}^k j^{m-1-i}=\frac{1}{m-i}(B_{m-i}(k+1)-B_{m-i}(0)),
$$
where $B_{m-i}$ is  the Bernoulli polynomial of degree $m-i$ 
(which has the leading coefficient equal to $1$), we finally obtain the desired result.
\end{proof}

\section{Linear Independence}
\label{sec:LinIndep}

\begin{theorem}
\label{thm:indep}
Let $F_1, \ldots, F_m$ be rational functions defined by~\eqref{eq:ratsyst} 
satisfying the conditions~\eqref{eq:deg1} and~\eqref{eq:deg2} and
such that $s_{i,i+1}\ne 0$, $i=1,\ldots,m-1$. Then, for $k\ne l$ and a nonzero 
vector $\vec{a}  \in \F^{m-1}$,  $Q_{k,l,\vec{a}}$ is 
a non-constant rational function.
\end{theorem}

\begin{proof}
The proof reduces to proving that 
$\deg (F_s^{(k)}-F_s^{(l)})>1$, where $s\le m-1$ is the smallest index such that $a_s\ne0$, as the variable $X_s$ does not appear in the polynomial 
$$
Q_{k,l,\vec{a}} - a_s  (F_s^{(k)}-F_s^{(l)}) = 
\sum_{i=s+1}^{m-1} a_i(F_i^{(k)}-F_i^{(l)}).
$$ 
If $e_s=1$, it is clear, as from Theorem~\ref{thm:deg}, for $k>l$ we have $\deg G_{s,k}>\deg G_{s,l}$. If $e_s=-1$, by~\eqref{eq:Fk-1}, we have
\begin{equation*}
\begin{split}
F_s^{(k)}-F_s^{(l)}&=\frac{X_sR_{s,k} +S_{s,k}}{X_sR_{s,k-1} +S_{s,k-1}}-\frac{X_sR_{s,l} +S_{s,l}}{X_sR_{s,l-1} +S_{s,l-1}} = \frac{U_{k,l,s}}{V
_{k,l,s}},
\end{split}
\end{equation*}
where 
\begin{equation*}
\begin{split}
U_{k,l,s}
= X_s^2(R_{s,k}&R_{s,l-1}-R_{s,k-1}R_{s,l})\\
+~&X_s(R_{s,k}S_{s,l-1}+S_{s,k}R_{s,l-1}-R_{s,k-1}S_{s,l} 
-S_{s,k-1}R_{s,l})\\
&+~S_{s,k}S_{s,l-1}-S_{s,k-1}S_{s,l}.
\end{split}
\end{equation*}
and 
$$
V_{k,l,s} = X_s^2R_{s,k-1}R_{s,k-1} +X_s (R_{s,k-1}S_{s,l-1} +R_{s,l-1}S_{s,k-1}) +
S_{s,k-1} S_{s,l-1}.
$$

Without loss of generality we may assume that $k>l$. 
Using~\eqref{eq:RS},  we obtain
\begin{equation*}
\begin{split}
R_{s,k}R_{s,l-1}-R_{s,k-1}R_{s,l}&=(G_s^{(k)}R_{s,k-2}+H_s^{(k)}R_{s,k-1})R_{s,l-1}\\
&\qquad\qquad\qquad\qquad-R_{s,k-1}(G_s^{(l)}R_{s,l-2}+H_s^{(l)}R_{s,l-1}),
\end{split}
\end{equation*}
and thus,
using Lemma~\ref{lem:LinTermSyst}   and Theorem~\ref{thm:deg}, we derive
\begin{equation}
\begin{split}
\label{eq:RRRR}
\deg (R_{s,k}R_{s,l-1}-R_{s,k-1}&R_{s,l})=\deg (H_s^{(k)}-H_s^{(l)})R_{s,k-1}R_{s,l-1}\\
&=\deg H_s^{(k)}R_{s,k-1}R_{s,l-1}>\deg R_{s,k-1}R_{s,l-1}>1
\end{split}
\end{equation}
for $k>l$, which concludes the proof.
\end{proof}

Note that, as in~\cite{OstShp2}, we can include $m$-term  linear  
combinations
$$
\overline Q_{k,l,\vec{a}}=\sum_{i=1}^{m } a_i(F_i^{(k)}-F_i^{(l)})
$$
with $\vec{a}\in\F^{m}$, but  in the case of $a_1 = \ldots = a_{m-1} = 0$, $a_m \ne 0$,
the nontriviality also depends on the divisibility of $k-l$ by
the multiplicative order of $g_m$. 

\section{Trajectory Lengths}

In this section we work over a finite field $\F_q$.

\begin{theorem}
\label{thm:traj}
Let $F_1, \ldots, F_m$ be rational functions defined by~\eqref{eq:ratsyst} satisfying the conditions~\eqref{eq:deg1} and~\eqref{eq:deg2} and
such that $s_{i,i+1}\ne 0$, $i=1,\ldots,m-1$. 
Then, for any $T\ge 1$ for all but $O(T^3q^{m-1})$ initial 
vectors $\vec{u}_0\in \F_q^m$, the trajectory length 
of the iterations~\eqref{eq:Vec} exceeds $T$.
\end{theorem}

\begin{proof} 
Let $\cU$ be the set of $\vec{u} = (u_1,\ldots, u_m) \in \F_q^m$ such that
$$
G_i(u_{i+1}, \ldots, u_{m}) = 0
$$
for some $i =1, \ldots, m-1$. 
Clearly $\# \cU = O(q^{m-1})$.

We now build a sequence of sets $\cU_k$, $k =0, 1, \ldots$,
recursively.

We put $\cU_0 = \cU$.

Assume that $\cU_0, \ldots, \cU_{k}$ are 
defined and let
$$
\cW_k =\cU_0 \cup \ldots \cup \cU_{k}.
$$
Then we  let $\cU_{k+1}$ be the set of  the initial values 
$$
\vec{u}_0\in \F_q^m \setminus \cW_k 
$$ 
such that for the corresponding sequence of vectors~\eqref{eq:Vec}
we have $\vec{u}_{k+1} \in \cU$. Inspecting~\eqref{eq:ratsyst}, 
we now see that, by our assumption, for any  
$\vec{v} \in \F_q^m$,
there is a unique preimage $\vec{u}   \in  \F_q^m \setminus  \cU$
under the map given by~\eqref{eq:ratsyst} (that is, with 
$\vec{v}=\vec{F}(\vec{u})$). 
In turn, we see that for any $\vec{v} \in \cU$ there is 
a unique corresponding 
initial value $\vec{u}_0\in \F_q^m \setminus \cW_k $ 
with  $\vec{u}_{k+1} = \vec{v}$. Thus $\# \cU_k = \# \cU$.

Since there are $O(q^{m-1})$ vectors $\vec{u}\in \F_q^m$
that contain a zero in at least one component, we see that the 
set $\cE_T$ of initial values for which, for some  integer $t \le T$,
the vector $\vec{u}_t$ has a zero component, satisfies 
\begin{equation}
\label{eq:ET}
\cE_T  = O(q^{m-1} + T \#\cU)  =  O(Tq^{m-1}).
\end{equation}

We see that if  a vector $\vec{u}_0\in \F_q^m \setminus \cE_T$
generates a trajectory of lengths $t \le T$ then
$\vec{u}_t=\vec{u}_s$ for some nonnegative integer $s<t$.

Now,  if $e_{m-1} = 1$, then  we remove the set $\cF_{T}$ of initial vectors  
$\vec{u}_0\in \F_q^m$ such that 
$$
G_{m-1,t}(\vec{u}_0) = G_{m-1,s}(\vec{u}_0)
$$
for some  integers $s$ and $t$ with
$T \ge t > s \ge 0$.
By Lemma~\ref{lem:LinTermSyst} we see that  
$G_{m-1,t}  - G_{m-1,s}$
is a nontrivial polynomial of degree $O(t)$.
Hence,
\begin{equation}
\label{eq:FT}
\# \cF_{T} = O\(\sum_{0 \le s < t \le T} t q^{m-1}\) =
O\(T^{3}q^{m-1}\).
\end{equation}

Furthermore, if $e_{m-1} = -1$, then we remove the set $\cF_{T}$ of initial vectors  
$\vec{u}_0\in \F_q^m$ such that   
$$R_{m-1,t}(\vec{u}_0)R_{m-1,s-1}(\vec{u}_0) = R_{m-1,t-1}(\vec{u}_0)R_{m-1,s}(\vec{u}_0)
$$
for some  integers $s$ and $t$ with
$T \ge t > s \ge 0$.
As in the proof of Theorem~\ref{thm:indep} 
(in particular, see~\eqref{eq:RRRR})
we note that Lemma~\ref{lem:LinTermSyst}
implies that $R_{m-1,t} R_{m-1,s-1} - R_{m-1,t-1}R_{m-1,s}$
is a nontrivial polynomial of degree $O(t)$.
Hence, again we obtain the bound~\eqref{eq:FT}.

We  remark that for $T \ge t > s \ge 0$,  
for any solution 
$\vec{u}_0=(u_{0,1},\ldots,u_{0,m})\in \F_q^m \setminus \cF_T$
to the equation
$$
\vec{F}^{(t)}(\vec{u}_{0}) =\vec{F}^{(s)}(\vec{u}_{0}),
$$
the component  $u_{0,m-1}$ is uniquely defined
by $u_{0,m}$. So there at most $q^{m-1}$ such solutions
for every fixed $t$ and $s$ with $T \ge t > s \ge 0$ 
and thus at most $T^2q^{m-1}$ for such $t$ and $s$.  
Combining this bound with~\eqref{eq:ET} and~\eqref{eq:FT}
we conclude the proof. \end{proof}

Clearly if the map $\vec{u} \mapsto \vec{F}(\vec{u})$
is a permutation, as for example, in~\cite{Ost1}, 
then all trajectories are purely periodic. So we always have
$s = 0$ in the argument of the proof of Theorem~\ref{thm:traj}.
This leads to a better estimate $O(T^2q^{m-1})$ on the number of 
initial values generating trajectories of length at most $T$.   

\section{Maximal Periods}

In this section we show that the periods of the rational function systems over $\F_q$ defined by~\eqref{eq:ratsyst} with $e_i=-1$ for all $i=1,\ldots,m$ are given by the orbit lengths of certain linear fractional transformations, also called M\"obius transformations. In particular, we describe the case when the systems~\eqref{eq:ratsyst} achieve maximal periods in their orbits. We also note that in~\cite[Theorem 6]{Ost3} there are given necessary and sufficient conditions for the system~\eqref{eq:ratsyst} to achieve maximal period in the case $e_i=1$ for all $i=1,\ldots,m$. 

We denote
$$
\widetilde{\vec{u}}_{0,i}=(u_{0,i+1},\ldots,u_{0,m})\in\F_q^{m-i}, \qquad i =1, \ldots, m-1.
$$

\begin{lemma}
\label{lem:kiter}
Let $F_1, \ldots, F_m\in \F_q[X_1,\ldots,X_m]$ be  as in~\eqref{eq:ratsyst} with $e_i=-1$ for all $i=1,\ldots,m$. Assume that the sequence generated by the lower $m-i$ rational functions $F_{i+1},\ldots,F_m$ in $\F_q^{m-i}$ is purely periodic with period $\tau_{i+1}$ for some $i=1,\ldots,m-1$. Then we have the following description for the $k\tau_{i+1}$-th iteration of $F_i$ on any initial vector $\vec{u}_0\in\F_q^{m}$:
\begin{equation}
\label{eq:iterSyst}
F_i^{(k\tau_{i+1})}(\vec{u}_0)=f_i^{(k)}(u_{0,i}),\qquad k\ge 1,
\end{equation}
where
$R_{i,\tau_{i+1}}$ and $S_{i,\tau_{i+1}}$ are defined by~\eqref{eq:RS} and $f_i$ is the M\" obius transformation in the variable $Y$,
$$
f_i(Y)=\frac{YR_{i,\tau_{i+1}}(\widetilde{\vec{u}}_{0,i})+S_{i,\tau_{i+1}}(\widetilde{\vec{u}}_{0,i})}{YR_{i,\tau_{i+1}-1}(\widetilde{\vec{u}}_{0,i})+S_{i,\tau_{i+1}-1}(\widetilde{\vec{u}}_{0,i})}.
$$
In particular, the orbit length of $F_i$ in $\vec{u}_0$ is given by the orbit length of $f_i$ in $u_{0,i}$.
\end{lemma}

\begin{proof}
We first note that the orbit length of $F_i$ is a multiple of $\tau_{i+1}$. Indeed, let $\tau_i$ be the orbit length of the system $F_i,\ldots,F_m$ in the initial vector $\vec{u}_0$. Then $\tau_i=\lcm(\tau_{i+1},\eta_i)$, where $\eta_i$ is the period ofthe sequence $\{u_{n,i}\}$ defined by the iterations of the polynomial $F_i$, and thus $\tau_i$ is a multiple of $\tau_{i+1}$. This shows that, in order to describe the period of $F_i,\ldots,F_m$ on the initial vector $\vec{u}_0$, it is enough to consider only $k\tau_{i+1}$-th iterations of $F_i$. 

By~\eqref{eq:Fk-1} we have
\begin{equation}
\label{eq:tau}
F_i^{(\tau_{i+1})}(\vec{u}_0) = \frac{u_{0,i}R_{i,\tau_{i+1}}(\widetilde{\vec{u}}_{0,i})+S_{i,\tau_{i+1}}(\widetilde{\vec{u}}_{0,i})}{u_{0,i}R_{i,\tau_{i+1}-1}(\widetilde{\vec{u}}_{0,i}) +S_{i,\tau_{i+1}-1}(\widetilde{\vec{u}}_{0,i})}=f_i(u_{0,i}).\end{equation}
Now, 
\begin{equation}
\label{eq:ktau}
\begin{split}
F_i^{(k\tau_{i+1})}&(\vec{u}_0)\\
 =&F_i^{(\tau_{i+1})}\(F_i^{((k-1)\tau_{i+1})}(\vec{u}_0),F_{i+1}^{((k-1)\tau_{i+1})}(\vec{u}_0),\ldots,F_m^{((k-1)\tau_{i+1})}(\vec{u}_0)\)\\
 =& F_i^{(\tau_{i+1})}\(F_i^{((k-1)\tau_{i+1})}(\vec{u}_0),u_{0,i+1},\ldots,u_{0,m}\)\\
 =& \frac{F_i^{((k-1)\tau_{i+1})}(\vec{u}_0)R_{i,\tau_{i+1}}(\widetilde{\vec{u}}_{0,i})+S_{i,\tau_{i+1}}(\widetilde{\vec{u}}_{0,i})}{F_i^{((k-1)\tau_{i+1})}(\vec{u}_0)R_{i,\tau_{i+1}-1}(\widetilde{\vec{u}}_{0,i}) +S_{i,\tau_{i+1}-1}(\widetilde{\vec{u}}_{0,i})}.
\end{split}
\end{equation}
To prove~\eqref{eq:iterSyst} we use induction over $k$. 
For $k=1$ it is clear. 

We now assume that the statement is 
true for $k-1$ and we prove it for $k$. 
Using~\eqref{eq:tau},~\eqref{eq:ktau} and the induction hypothesis, we derive 
\begin{equation*}
\begin{split}
F_i^{(k\tau_{i+1})}(\vec{u}_0)
&=\frac{f_i^{(k-1)}(u_{0,i})R_{i,\tau_{i+1}}(\widetilde{\vec{u}}_{0,i})+S_{i,\tau_{i+1}}(\widetilde{\vec{u}}_{0,i})}{f_i^{(k-1)}(u_{0,i})R_{i,\tau_{i+1}-1}(\widetilde{\vec{u}}_{0,i}) +S_{i,\tau_{i+1}-1}(\widetilde{\vec{u}}_{0,i})}=f_i^{(k)}(u_{0,i}),
\end{split}
\end{equation*}
which concludes  the proof.
\end{proof}

\begin{lemma}
\label{lem:prod}
Let $\cF=\{F_1,\ldots,F_m\}$ be a system of 
polynomials over $\F_q$ defined 
by~\eqref{eq:ratsyst}. Let $i=1,\ldots,m$ such that $e_i=-1$ in the system ~\eqref{eq:ratsyst}. Then we have
$$
R_{i,k}S_{i,k-1}-R_{i,k-1}S_{i,k}=(-1)^kG_iG_i^{(2)}\ldots G_i^{(k)},
$$
where $R_{i,k},S_{i,k}$ are defined by~\eqref{eq:RS}.
\end{lemma}

\begin{proof}
We use induction on $k$. If $k=1$, by~\eqref{eq:RS}, we get
$$
R_{i,1}S_{i,0}-R_{i,0}S_{i,1}=-G_i.
$$
We assume the statement true for $k$ and we prove it for $k+1$. By~\eqref{eq:RS} and the induction hypothesis we have
\begin{equation*}
\begin{split}
R_{i,k+1}S_{i,k}&-R_{i,k}S_{i,k+1}\\
 =&(G_i^{(k+1)}R_{i,k-1}+H_i^{(k+1)}R_{i,k})S_{i,k}-R_{i,k}(G_i^{(k+1)}S_{i,k-1}+H_i^{(k+1)}S_{i,k})\\
 =&-G_i^{(k+1)}(R_{i,k}S_{i,k-1}-R_{i,k-1}S_{i,k})\\
 =&(-1)^{k+1}G_iG_i^{(2)}\ldots G_i^{(k)}G_i^{(k+1)}
\end{split}
\end{equation*}
and thus we conclude the proof.
\end{proof}

As usual, we say that a polynomial $f\in\F_q[X]$ of degree $d\ge 1$ is   {\it primitive\/}  if it is the minimal polynomial over $\F_q$ of a primitive element of $\F_{q^d}$ (that is, an element of 
multiplicative order $q^d-1$), see~\cite{LN}.

Next, we present necessary and sufficient conditions for the system~\eqref{eq:ratsyst} to achieve maximal period over the prime field $\F_p$. 

Using~\cite[Lemma 2]{Ost3} which holds for the functions $F_i$ in the system~\eqref{eq:ratsyst} for which $e_i=1$, we have the following analogue of~\cite[Lemma 5]{Ost3}
(with an almost identical proof which we do not present here). 

\begin{lemma}
\label{lem:iterH}
Let $\cF=\{F_1,\ldots,F_m\}$ be a system of 
polynomials over $\F_p$ defined 
by~\eqref{eq:ratsyst}. Let the index $1\le i\le m$ such that $e_i=1$ and assume that the period of the sequence generated by the lower $m-i$ polynomials $F_{i+1},\ldots,F_m$ in $\F_p^{m-i}$ is $p^{m-i}$ and that $G_{i,p^{m-i}}(\widetilde{\vec{u}}_{0,i})=1$. Then, for 
the rational functions $H_{i,p^{m-i}}$ defined by~\eqref{eq:GH}, we have
$$
H_{i,p^{m-i}}(\widetilde{\vec{u}}_{0,i})=\sum_{\vec{v}\in\F_p^{m-i}}R_i(\vec{v}),
$$
where
\begin{equation}
\label{eq:R_i}
R_i\equiv H_iG_i^{(2)}\ldots G_i^{(p^{m-i})}.\end{equation}
\end{lemma}
 
Now, using Lemma~\ref{lem:iterH} and~\cite[Theorem 6]{Ost3} for $F_i$ with $e_i=1$ in the system~\eqref{eq:ratsyst}, we have the following result.
 
We recall that the sets $I_+$ and $I_-$ are given by~\eqref{eq:I+-}.  
 
\begin{theorem}
Let $\cF=\{F_1,\ldots,F_m\}$ be a system of 
polynomials over $\F_p$ defined 
by~\eqref{eq:ratsyst}.
Then  the sequence $\{\vec{u}_n\}$ generated by~\eqref{eq:Vec} is purely periodic
with period $\tau=p^m$ if and only if the following 
conditions are satisfied
\begin{enumerate}
\item for every $i\in I_+$, $i < m$, we have
$$
\prod_{\vec{v}\in\F_p^{m-i}}G_i(\vec{v})=1 \quad \text{and} \quad \sum_{\vec{v}\in\F_p^{m-i}}R_i(\vec{v})\ne 0;
$$
\item for every $i\in I_-$, $i < m$, we have:
\begin{enumerate}
\item if $R_{i,p^{m-i}-1}({\vec{u}}_{0})=0$, then
$$
R_{i,p^{m-i}}({\vec{u}}_{0})=S_{i,p^{m-i}-1}({\vec{u}}_{0}) \quad 
\text{and} \quad 
S_{i,p^{m-i}}({\vec{u}}_{0})S_{i,p^{m-i}-1}({\vec{u}}_{0})\ne 0;
$$
\item 
if 
$R_{i,p^{m-i}-1}({\vec{u}}_{0})\ne 0$,  then
$$
X^2-\frac{R_{i,p^{m-i}}({\vec{u}}_{0})}{R_{i,p^{m-i}-1}({\vec{u}}_{0})}X-\frac{\prod_{\vec{v}\in\F_p^{m-i}}G_1(\vec{v})}{R_{i,p^{m-i}-1}({\vec{u}}_{0})}
$$
is a primitive polynomial over $\F_p$;
\end{enumerate}

\item if $m \in I_+$, then $ g_m=1$;
\item if $m \in I_-$, then $X^2-h_mX-g_m$
is a primitive polynomial over $\F_p$. 
\end{enumerate}
\end{theorem}

\begin{proof}
We prove the result by induction on $m$. 
For $m=1$, if $m\in I_+$, then it's clear that the period $p$ is achieved if and only if $g_m=1$ and $h_m\in\F_p^*$. If $m\in I_-$, we have 
$$
F_1=g_1X_1^{-1}+h_1,
$$
which, by~\cite[Theorem~1]{FN}, has maximal period $p$ if and only if the polynomial $X^2-h_1X-g_1$ is a primitive polynomial over $\F_p$.

We assume the statement true for the first $m-1$ variables and we want to prove it for $m$. 
Let  the sequence $\{\widetilde{\vec{u}}_{n,1}\}=\{(u_{n,2},\ldots,u_{n,m})\}$ 
be defined by the last $m-1$ components of the vectors in the sequence $\{\vec{u}_n\}$.
By the induction hypothesis we know that the period $\widetilde{\tau}$ of the sequence 
$\{\widetilde{\vec{u}}_{n,1}\}$ is $p^{m-1}$, and taking into account the first remark 
in the proof of Lemma~\ref{lem:kiter}, we see that the period of $\{\widetilde{\vec{u}}_{n,1}\}$ 
is of the form $k p^{m-1}$, for some $1\le k\le q$. 
Thus, proving the maximality of the  period of 
$\{\vec{u}_n\}$ reduces to proving that $k=q$. 
 We note that the representation of $F_1^{(k)}$ given by Lemma~\ref{lem:LinTermSyst} does not depend how we choose $e_{2},\ldots,e_m$ in the functions $F_{2},\ldots,F_m$, but only the functions $G_{1,k},H_{1,k}$ or $G_1^{(k)},H_1^{(k)}$  if $1\in I_+$ or $1\in I_-$, respectively. 
 Thus, the representation of $F_1^{(k)}$ given by Lemma~\ref{lem:LinTermSyst} is the same regardless if $i_j\in I_+$ or $i_j\in I_-$ for $2\le j\le m$.
 
Thus, the case $1\in I_+$ follows identically as in the proof of~\cite[Theorem]{Ost3}, and we do not repeat it here.

We consider now the case $1\in I_-$. 
 By Lemma~\ref{lem:kiter} we have
$$
F_1^{(kp^{m-1})}(\vec{u}_0)=f_1^{(k)}(u_{0,1}),\quad k\ge 1,
$$
where
$$
f_1(Y)=\frac{YR_{1,p^{m-1}}(\widetilde{\vec{u}}_{0,1})+S_{1,p^{m-1}}(\widetilde{\vec{u}}_{0,1})}{YR_{1,p^{m-1}-1}(\widetilde{\vec{u}}_{0,1})+S_{1,p^{m-1}-1}(\widetilde{\vec{u}}_{0,1})},
$$
and thus the maximal period of the sequence generated by the iterations of $F_1$ is given by the case when $f_1$ achieves maximal orbit length in $u_{0,1}$. 

We distinguish now two cases. If $R_{1,p^{m-1}-1}(\widetilde{\vec{u}}_{0,1})=0$, then we note that $S_{1,p^{m-1}-1}(\widetilde{\vec{u}}_{0,1})\ne 0$, as otherwise $f_1=0$. Thus, we have the case of linear generator
$$
f_1(Y)=\frac{R_{1,p^{m-1}}(\widetilde{\vec{u}}_{0,1})}{S_{1,p^{m-1}-1}(\widetilde{\vec{u}}_{0,1})}Y+\frac{S_{1,p^{m-1}}(\widetilde{\vec{u}}_{0,1})}{S_{1,p^{m-1}-1}(\widetilde{\vec{u}}_{0,1})}.
$$
The maximal period is achieved in this case if and only if
$$
\frac{R_{1,p^{m-1}}(\widetilde{\vec{u}}_{0,1})}{S_{1,p^{m-1}-1}(\widetilde{\vec{u}}_{0,1})}=1,\quad \frac{S_{1,p^{m-1}}(\widetilde{\vec{u}}_{0,1})}{S_{1,p^{m-1}-1}(\widetilde{\vec{u}}_{0,1})}\ne 0,
$$ 
which is equivalent to
$$
R_{1,p^{m-1}}(\widetilde{\vec{u}}_{0,1})=S_{1,p^{m-1}-1}(\widetilde{\vec{u}}_{0,1}),\quad S_{1,p^{m-1}}(\widetilde{\vec{u}}_{0,1})\ne 0.
$$

We consider now the case of $R_{1,p^{m-1}-1}(\widetilde{\vec{u}}_{0,1})\ne 0$. We note that in this case $f_1$ achieves maximal period if and only if, under a linear transformation, it has the same property. Taking into account that $R_{1,p^{m-1}-1}(\widetilde{\vec{u}}_{0,1})\ne 0$, we can make the linear transformation 
$$Y\to R_{1,p^{m-1}-1}(\widetilde{\vec{u}}_{0,1})^{-1}Y-R_{1,p^{m-1}-1}(\widetilde{\vec{u}}_{0,1})^{-1}S_{1,p^{m-1}-1}(\widetilde{\vec{u}}_{0,1})$$ 
and obtain the following inversive generator which, by a slightly abuse of notation, we denote also by $f_1$,
\begin{equation*}
\begin{split}
f_1(Y)=\frac{-R_{1,p^{m-1}}(\widetilde{\vec{u}}_{0,1})S_{1,p^{m-1}-1}(\widetilde{\vec{u}}_{0,1})+R_{1,p^{m-1}-1}(\widetilde{\vec{u}}_{0,1})S_{1,p^{m-1}}(\widetilde{\vec{u}}_{0,1})}{R_{1,^{m-1}}(\widetilde{\vec{u}}_{0,1})}Y^{-1}\quad&\\
    +\frac{R_{1,p^{m-1}}(\widetilde{\vec{u}}_{0,1})}{R_{1,p^{m-1}-1}(\widetilde{\vec{u}}_{0,1})}&.
\end{split}
\end{equation*}

Applying now Lemma~\ref{lem:prod} we have
\begin{equation*}
\begin{split}
R_{1,p^{m-1}}(\widetilde{\vec{u}}_{0,1})S_{1,p^{m-1}-1}(\widetilde{\vec{u}}_{0,1})-&R_{1,p^{m-1}-1}(\widetilde{\vec{u}}_{0,1})S_{1,p^{m-1}}(\widetilde{\vec{u}}_{0,1})\\
& =-G_1(\widetilde{\vec{u}}_{0,1})G_1^{(2)}(\widetilde{\vec{u}}_{0,1})\ldots G_1^{(p^{m-1})}(\widetilde{\vec{u}}_{0,1}).
\end{split}
\end{equation*}

Let $\widetilde{\cF}=\{F_{2},\ldots,F_m\}$. Now, as the period induced by $\widetilde{\cF}$ is $p^{m-1}$, the elements
$$
\widetilde{\vec{u}}_{0,1},\widetilde{\cF}(\widetilde{\vec{u}}_{0,1}),\ldots,\widetilde{\cF}^{(p^{m-1}-1)}(\widetilde{\vec{u}}_{0,1})
$$
are all the distinct elements of $\F_p^{m-i}$, and thus we obtain
\begin{equation*}
\begin{split}
G_{1}(\widetilde{\vec{u}}_{0,1})G_{1}^{(2)}&(\widetilde{\vec{u}}_{0,1})\ldots G_{1}^{(p^{m-1})}(\widetilde{\vec{u}}_{0,1})\\
&=G_{1}(\widetilde{\vec{u}}_{0,1})G_{1}(\widetilde{\cF}(\widetilde{\vec{u}}_{0,1}))\ldots G_{1}(\widetilde{\cF}^{(p^{m-1}-1)}(\widetilde{\vec{u}}_{0,1}))=\prod_{\vec{v}\in\F_p^{m-1}}G_1(\vec{v}).
\end{split}
\end{equation*}
This concludes that
$$
f_1(Y)=\frac{\prod_{\vec{v}\in\F_p^{m-1}}G_1(\vec{v})}{R_{1,p^{m-1}}(\widetilde{\vec{u}}_{0,1})}Y^{-1}+\frac{R_{1,p^{m-1}}(\widetilde{\vec{u}}_{0,1})}{R_{1,p^{m-1}-1}(\widetilde{\vec{u}}_{0,1})}.
$$
Applying now~\cite[Theorem~1]{FN}, we know that $f_1$ achieves maximal period $q$ if and only if the polynomial 
$$X^2-\frac{R_{1,p^{m-1}}(\widetilde{\vec{u}}_{0,1})}{R_{1,p^{m-1}-1}(\widetilde{\vec{u}}_{0,1})}X-\frac{\prod_{\vec{v}\in\F_p^{m-1}}G_1(\vec{v})}{R_{1,p^{m-1}-1}(\widetilde{\vec{u}}_{0,1})}$$ 
is a primitive polynomial over $\F_p$.
\end{proof}
 
\section*{Acknowledgment}

During the preparation of this paper,  
A.~O. was supported in part by 
the Swiss National Science Foundation   Grant~PBZHP2--133399 
and I.~S. by
the  Australian Research Council 
Grant~DP1092835.

\end{document}